\documentclass[pandoc, label-section=modern, 12pt, twoside, leqno]{amsart}
\usepackage[top=30truemm,bottom=30truemm,left=20truemm,right=20truemm]{geometry}
\usepackage{mathtools, amssymb, amsthm, fancyhdr, tcolorbox, environ, xcolor, mathrsfs}
\usepackage[inline]{enumitem}
\numberwithin{equation}{section}

\theoremstyle{definition}
\newtheorem{theorem}{Theorem}[section]
\newtheorem{lemma}[theorem]{Lemma}
\newtheorem{proposition}[theorem]{Proposition}

\newtheorem{remark}[theorem]{Remark}

\newcommand{\C}{\mathbb{C}}
\newcommand{\N}{\mathbb{N}}

\newcommand{\R}{\mathbb{R}}

\mathtoolsset{showonlyrefs=true}

\title[Cauchy problem of Zakharov system]{Upper bound of higher-order Sobolev norms for Zakharov system in one space dimension}
\date{}
\author[N. Kobayashi]{Nobutatsu Kobayashi*}
\thanks{* Department of Mathematics, Graduate School of Science, 
Tokyo University of Science\\
1-3 Kagurazaka, Shinjuku-ku, Tokyo 162-8601, Japan\\
Email : 1122509@ed.tus.ac.jp
}

\begin{document}
\begin{abstract}
    We study the Cauchy problem for the Zakharov system in one space dimension with the Diriclet boundary conditions. We establish the global well-posedness and the growth of higher-order Sobolev norms of solutions to the Zakharov system by using the modified energy method.


\end{abstract}

\maketitle



  

\section{Introduction}
In this paper we consider the initial-boundary value problem:
\begin{equation}
	\left\{\begin{alignedat}{3}
      &i\partial_t u+\partial_x^2 u=vu,&&(t,x)\in\R\times I,\\
      &\partial_t^2 v-\partial_x^2 v=\partial_x^2(|u|^2),&&(t,x)\in\R\times I,\\
      &(u(0,x),v(0,x),\partial_t v(0,x))=(u_0(x),v_0(x),v_1(x)),\ && x\in I,\\
      &u(t,x)=v(t,x)=0,&& t\in\R, x\in\partial I,
       \end{alignedat}\right.
       \label{eq:(1.1)}
\end{equation}
where $I\subset\R$ is an open interval and $u:\R\times I\to\C$, $v:\R\times I\to\R$. We rewrite \eqref{eq:(1.1)} as the following first-order problem:
\begin{equation}
	\left\{\begin{alignedat}{4}
      &\partial_t u=i\partial_x^2 u-ivu,&&(t,x)\in\R\times I,\\
      &\partial_t v=-\partial_x w,&&(t,x)\in\R\times I,\\
      &\partial_t w=-\partial_x v-\partial_x(|u|^2),&&(t,x)\in\R\times I,\\
      &(u(0,x),v(0,x),w(0,x))=(u_0(x),v_0(x),w_0(x)),\ &&x\in I,\\
      &u(t,x)=v(t,x)=w(t,x)=0,&& t\in\R, x\in\partial I,
\end{alignedat}\right.
       \label{eq:(1.2)}
\end{equation}
where $w:\R\times I\to\R$.

\vspace{3mm}
The solution of \eqref{eq:(1.2)} formally obeys the following conservative quantities:
\begin{align*}
    &\text{Mass} : M(u(t)):=\|u(t)\|_{L^2}^2,\\
    &\text{Energy} : E(U(t)):=\|\partial_x u(t)\|_{L^2}^2+\frac{1}{2}(\|v(t)\|_{L^2}^2+\|w(t)\|_{L^2}^2)+(v(t),|u(t)|^2)_{L^2},
\end{align*}
where $U=(u,v,w)$ and 
\begin{equation}
    (u,v)_{L^2}=\mathrm{Re}\displaystyle\int_{I}u(x)\overline{v(x)}dx.
\end{equation}
We define the energy space $X=H_0^1(I)\times L^2(I)\times L^2(I)$.

Physically, \eqref{eq:(1.1)} and \eqref{eq:(1.2)} were first introduced by Zakharov \cite{MR15} as a model which describes long-wavelength Langmuir turbulence in a plasma. In this system, $u$ denotes the envelope of the electric field and $v$ is the deviation of the ion density from its equilibrium. The system \eqref{eq:(1.2)} was given by Gibbons, Thornhill, Wardrop and Haar \cite{MR16} from a Langmuir formalism.

The initial value problem for the Zakharov system in $\R^N$ has been well studied and started with the result of Sulem--Sulem \cite{MR11}. Sulem--Sulem showed the existence and uniqueness of local solutions $(u,v,w)\in L^{\infty}([0,T];H^m(\R^N)\times H^{m-1}(\R^N)\times H^{m-1}(\R^N))$ for $m\ge 3$ with $N\le 3$ and of global solutions in one space dimension. Later, Added--Added \cite{MR12} proved the unique global solution $(u,v,w)\in L_{loc}^{\infty}(\R;H^m(\R^2)\times H^{m-1}(\R^2)\times H^{m-1}(\R^2))$ for $m\ge 3$. Note that these results are proved by the energy method, which does not rely on the Fourier transform. Moreover, Ozawa--Tsutsumi \cite{MR2} obtained the existence and uniqueness of local solutions $(u,v,w)\in C([0,T];H^2(\R^N)\times H^1(\R^N)\times H^1(\R^2))$ for $N\le 3$, by using the elliptic regularity theory. Later, Bourgain--Colliander \cite{MR3} obtains the existence and uniqueness of local solutions $(u,v,w)\in C([0,T];\\H^1(\R^2)\times L^2(\R^2)\times L^2(\R^2))$ for the Zakharov system in the energy space with $N=2$. The method of the proof depends on the Fourier restricted norm method. Furthermore, in \cite{MR4}, the results in arbitrary dimensions are discussed. On the other hand, there are few studies in general domains. One of the reasons for this is that the Strichartz estimates and the Bourgain's method cannot be applied to a general domain because the Fourier transform cannot be discussed in general domains. In \cite{MR1}, Ozawa and Tomioka obtain the unique global solution in a general domain for the Zakharov system in two space dimensions.

We verify the global result in one space dimension.

\begin{theorem} \label{theorem1}
    Let $(u_0,v_0,w_0)\in (H^2\cap H_0^1)(I)\times H_0^1(I)\times H_0^1(I)$. Then there exists the unique solution $(u,v,w)$ of \eqref{eq:(1.2)} such that
    \begin{equation}
        (u,v,w)\in C(\R;(H^2\cap H_0^1)(I)\times H_0^1(I)\times H_0^1(I))\cap C^1(\R;L^2(I)\times L^2(I)\times L^2(I))\label{eq:1.3}.
    \end{equation}
    Moreover, the following properties hold:
    \begin{enumerate}
        \item $M(u(t))=M(u_0),\hspace{1mm}E(U(t))=E(U(0))$ for all $t\in\R$.
        \item $(u,v,w)$ depends continuously on $(u_0,v_0,w_0)$ in the following sense; if
        \begin{equation}
        (u_0^n,v_0^n,w_0^n)\to (u_0,v_0,w_0)\
    \mathrm{in}\ H^2(I)\times H^1(I)\times H^1(I)\ \text{as}\ n\to\infty
        \end{equation}
    and if $(u_n,v_n,w_n)$ is the corresponding solution of \eqref{eq:(1.2)}, then
    \begin{equation}
    (u_n,v_n,w_n)\to (u,v,w)\in C([-T,T];H^s(I)\times H^{\sigma}(I)\times H^{\sigma}(I))
    \end{equation}
    for any $(s, \sigma)\in[0,2)\times [0,1)$.
    \end{enumerate}
\end{theorem}
\begin{remark}
\noindent (1) We can also show the statement \eqref{eq:1.3} for $N=2$ in the same argument. This means that we obtain the unique solution of \eqref{eq:(1.2)} more regular than in \cite{MR1}; the first component $u$ of the solution obtained in \cite{MR1} only belongs to $C_w(\R;(H^2\cap H_0^1)(I))$. Note that this argument does not rely on dimensions.\\
\noindent (2) Boundedness of the interval $I$ is not assumed because we do not use any compactness argument.
\noindent (3) Since $H^1$ is embedded in $L^{\infty}$, we do not need the Yudovitch argument based on the Sobolev inequalities in the critical case, which is a key idea in two space dimensions \cite{MR1}. 
\end{remark}


Growth estimates of higher-order Sobolev norms of the solution to Schr\"{o}dinger equations in $\R^N$ and the torus $\mathbb{T}^N$ have been investigated (e.g. Bourgain \cite{MR5}). Some results on the Zakharov system have been also obtained. For example, Colliander-Staffilani \cite{MR14}
showed that
\begin{align}
    \|u(t)\|_{H^s}\le C|t|^{(s-1)+}
\end{align}
 holds for solutions to the system with small initial data, where $a$+$\hspace{1mm}(a\in\R)$ denotes $a+\varepsilon$ for any $\varepsilon>0$. There are some studies on growth estimates of the Sobolev norms of solutions to the Zakharov system in a general domain (e.g. \cite{MR1,MR6}). Especially in \cite{MR1}, Ozawa and Tomioka obtain that the higher-order Sobolev norms grow at most exponentially in two space dimensions:
\begin{align}
    \|u(t)\|_{H^2}+\|v(t)\|_{H^1}+\|w(t)\|_{H^1}\le C\exp{(C|t|)},
\end{align}
where $C$ is a positive constant depends only on $\|u_0\|_{H^2}$, $\|v_0\|_{H^1}$ and $\|w_0\|_{H^1}$.

However, there seem to be no results on the higher-order Sobolev norms in one dimension. Hence, we studied this and obtained the following result.

\begin{theorem}
    Let $(u,v,w)$ be the solution of \eqref{eq:(1.2)} with $(u_0,v_0,w_0)\in (H^2\cap H_0^1)(I)\times H_0^1(I)\times H_0^1(I)$ obtained in Theorem \ref{theorem1}. Then there exists $C=C(\|u_0\|_{H^2},\|v_0\|_{H^1},\|w_0\|_{H^1})>0$ such that
    \begin{equation}
        \|u(t)\|_{H^2}+\|v(t)\|_{H^1}+\|w(t)\|_{H^1}\le C(1+|t|^2)\hspace{2mm}\mathrm{for}\hspace{1mm}\mathrm{all}\hspace{2mm}t\in\R.
    \end{equation}
\end{theorem}

This paper is organized as follows. In Section $2$, we construct the solutions for the approximate problem and obtain uniform estimates for them. Next, we establish the uniform bounds in $H^2\times H^1\times H^1$ under the uniform bounds in $H^1\times L^2\times L^2$. In Section $4$, we prove that the sequence of approximate solutions forms a Cauchy sequence in $C([-T,T];H_0^1(I)\times L^2(I)\times L^2(I))$ by introducing a modified energy. Finally, we construct the solution of \eqref{eq:(1.2)} and argue the regularity with respect to the first component $u$ of the solution based on \cite{MR7}.

\section{Construction of approximate solutions}
In this section, we consider the following regularized system
\begin{equation}
	\left\{\begin{alignedat}{2}
      &\partial_t u_n=i\partial_x^2 u_n-iJ_n(J_n v_n\cdot J_n u_n),\\
      &\partial_t v_n=-\partial_x w_n,\\
      &\partial_t w_n=-\partial_x v_n-\partial_x J_n(|J_n u_n|^2)\\
       \end{alignedat}\right.
       \label{eq:(2.1)}
\end{equation}
with initial conditions
\begin{equation}
    (u_n(0),v_n(0),w_n(0))=(J_n u_0,J_n v_0,J_n w_0)\in (H^2\cap H_0^1)(I)\times H_0^1(I)\times H_0^1(I),
\end{equation}
where $J_n$ is the Yosida regularization defined by
\begin{equation}
    J_n=\left(1-\frac{1}{n}\partial_x^2\right)^{-1}.
\end{equation}
Note that $\partial_x^2$ is a self-adjoint operator in $L^2(I)$ with the dense domain $(H^2\cap H_0^1)(I)$. We summarize several properties of $J_n$ needed for the proof of the main theorems.
\begin{lemma}
    \label{lemma_2_1}
    Let $X$ be either of the spaces $L^2(I)$, $H_0^1(I)$, $(H^2(I)\cap H_0^1)(I)$ and let $X^*$ be its dual space. Then, the following properties hold:

    \vspace{1mm}
    \noindent (1) $\langle J_n f,g\rangle_{X,X^*}=\langle f,J_n g\rangle_{X,X^*}\hspace{2mm}\mathrm{for}\hspace{2mm}f\in X,\hspace{1mm}g\in X^*$.
    
    \vspace{1mm}
    \noindent (2) $J_n$ are bounded self-adjoint operators in $L^2(I)$ with $J_n(L^2(I))=(H^2(I)\cap H_0^1)(I)$.

    \vspace{1mm}
    \noindent (3) $\|J_n\|_{\mathcal{L}(X)}\le 1$.

    \vspace{1mm}
    \noindent (4) For any $u\in X$,
        $J_n u\to u\hspace{2mm}\mathrm{in}\hspace{2mm}X\hspace{2mm}\mathrm{as}\hspace{2mm}n\to\infty$.

    \vspace{1mm}
    \noindent (5) For any $u\in L^2(I)$,
    \begin{equation}
        \|\partial_x J_n u\|_{L^2}\le n\|u\|_{L^2},\quad \|\partial_x^2 u\|_{L^2}\le n\|u\|_{L^2}.
    \end{equation}
\end{lemma}
For the proof, see \cite{MR8,MR1,MR9}.
\begin{lemma}(\cite{MR1})
\label{lemma_2_2}
    For any $m,n\in\mathbb{N}\hspace{1mm}(m>n)$, $u\in H_{0}^1(I)$, 
    \begin{align*}
        &\|(J_m-J_n)u\|_{L^2}\le \frac{1}{n^{1/2}}\|\partial_x u\|_{L^2},\\
        &\|(J_m-J_n)u\|_{L^4}^2\le C\frac{1}{n^{1/2}}\|\partial_x u\|_{L^2}^2.
    \end{align*}
\end{lemma}

Now, we define an approximate energy as
\begin{equation}
    E_n(U):= \|\partial_x u\|_{L^2}+\frac{1}{2}(\|v\|_{L^2}+\|w\|_{L^2})+(J_n v, |J_n u|^2)_{L^2}.
\end{equation}
A standard calculation shows the conservation laws of mass and energy for the approximate problem.
\begin{lemma}
    For all $t\in[-T_n,T_n]$,

    \vspace{1mm}
    \noindent (i) $M(u_n(t))=M(u_n(0))$

    \vspace{1mm}
    \noindent (ii) $E_n(U_n(t))=E_n(U_n(0))$.
\end{lemma}
We set $U_n(t)=(u_n(t),v_n(t),w_n(t))$ so that \eqref{eq:(2.1)} becomes
\begin{equation}
    \partial_t U_n(t) = AU_n(t) + G_n(U_n(t))\label{eq:(2.2)},
\end{equation}
where
\begin{equation}
A=
\begin{pmatrix*}[c]
    i\partial_x^2 &0 &0 \\
    0 &0 &-\partial_x \\
    0 &-\partial_x &0
\end{pmatrix*},
\quad
    G_n(U)=
    \begin{pmatrix*}[c]
        -iJ_n(J_n v_n\cdot J_n u_n) \\
        0 \\
        -\partial_x J_n(|J_n u_n|^2)
    \end{pmatrix*}.
\end{equation}

Since $A$ is a skew-adjoint operator in $X=H_0^1(I)\times L^2(I)\times L^2(I)$ with the dense domain $D(A)=(H^3\cap H_0^1)(I)\times H_0^1(I)\times H_0^1(I)$, and $G_n$ is Lipschitz continuous on bounded subsets of $X$ and $U_n(0)\in D(A)$, there exists $T_n>0$ and the unique local solution $U_n=(u_n,v_n,w_n)\in C([-T_n,T_n];D(A))\cap C^1([-T_n,T_n];X)$ to the initial value problem \eqref{eq:(2.2)}.

For these approximate solutions, we obtain the following uniform estimate.

\begin{lemma}
There exists $M_1>0$, which depends on $\|u_0\|_{H^1},\|v_0\|_{L^2},\|w_0\|_{L^2}$, such that
    \begin{equation}
        \sup_{n\in\N,t\in\R}(\|u_n(t)\|_{H^1}+\|v_n(t)\|_{L^2}+\|w_n(t)\|_{L^2})\le M_{1}.
    \end{equation}
\end{lemma}
\begin{proof}
    Since the inner product term of the approximate energy $E_n$ is estimated as
    \begin{align*}
        |(J_n v_n,|J_n u_n|^2)_{L^2}|
        &\le \|J_n v_n\|_{L^2}\|J_n u_n\|_{L^4}^2\\
        &\le C\|v_n\|_{L^2}\|u_n\|_{L^2}^{3/2}\|\partial_x u_n\|_{L^2}^{1/2}\\
        &\le (C\|u_0\|_{L^2}^3+\frac{1}{2}\|\partial_x u_n\|_{L^2})\|v_n\|_{L^2}\\
        &\le \frac{1}{2}\|\partial_x u_n\|_{L^2}^2+\frac{1}{4}\|v_n\|_{L^2}^2+C(\|u_0\|_{L^2}),
    \end{align*}
    we obtain
    \begin{align*}
        &\frac{1}{2}\|\partial_x  u_n\|_{L^2}^2+\frac{1}{4}\|v_n\|_{L^2}^2+\frac{1}{2}\|w_n\|_{L^2}^2-C(\|u_0\|_{L^2})\\
        &\le E_n(U_n(t))=E_n(U_n(0))\\
        &\le \|\partial_x u_0\|_{L^2}^2+\frac{1}{2}(\|v_0\|_{L^2}^2+\|w_0\|_{L^2}^2)+C\|v_0\|_{L^2}\|u_0\|_{L^2}^{2/3}\|\partial_x u_0\|_{L^2}^{1/2}
    \end{align*}
    by using Lemma $2.3$ and
    \begin{align*}
        |(J_{n}^{2}v_0,|J_{n}^{2}u_0|^2)_{L^2}|
        &\le \|J_n v_0\|_{L^2}\|J_n u_0\|_{L^4}^2\\
        &\le C\|v_0\|_{L^2}\|u_0\|_{L^2}^{3/2}\|\partial_x u_0\|_{L^2}^{1/2}.
    \end{align*}
    Therefore, we have
    \begin{align*}
        &\sup_{n\in\N,t\in\R}(\|u_n(t)\|_{L^2}^2+\|\partial_x u_n(t)\|_{L^2}^2+\|v_n(t)\|_{L^2}^2+\|w_n(t)\|_{L^2}^2)\\
        &\le \|\partial_x u_0\|_{L^2}^2+\frac{1}{2}(\|v_0\|_{L^2}^2+\|w_0\|_{L^2}^2)+C\|v_0\|_{L^2}\|u_0\|_{L^2}^{2/3}\|\partial_x u_0\|_{L^2}^{1/2}=:M_1^2.
    \end{align*}
\end{proof}

Lemma $2.4$ implies that the energy norm is estimated a priori, then the local solutions for \eqref{eq:(2.1)} can be extended globally in time.

\section{Uniform boundedness of the sequence of approximate solutions}
In this section, we establish the uniform $H^2\times H^1\times H^1$ estimates of the sequence of solutions for the regularized system.

\begin{proposition}
\label{prop_3_1}
For all $t\in\R$,
    \begin{equation}
        \sup_{n\in\N}\hspace{1mm}(\|u_n(t)\|_{H^2}+\|v_n(t)\|_{H^1}+\|w_n(t)\|_{H^1})\le C(M_1)(1+|t|^2).
    \end{equation}
\end{proposition}
\begin{proof}
    We introduce the higher-order energy $F_n$ to the approximate problem \eqref{eq:(2.1)} defined by
    \begin{equation}
        F_n(t)=\|\partial_t u_n(t)\|_{L^2}^2+\frac{1}{2}(\|\partial_x v_n(t)\|_{L^2}^2+\|\partial_x w_n(t)\|_{L^2}^2).
    \end{equation}
    Calculating the time derivative of $F_n$,
    \begin{align}
        \frac{d}{dt}F_n(t)
        &= 2\langle \partial_t^2 u_n, \partial_t u_n\rangle +\langle\partial_t \partial_x v_n, \partial_x v_n\rangle+\langle\partial_t \partial_x w_n, \partial_x w_n\rangle \label{eq:3.1}\\
        &= 2\langle \partial_t u_n(i\partial_x^2 u_n-iJ_n(J_n v_n\cdot J_n u_n)), \partial_t u_n\rangle +\langle-\partial_x w_n,\partial_x^2 v_n\rangle\\
        &\quad +\langle\partial_x w_n,\partial_x^2 v_n+J_n \partial_x^2 |J_n u_n|^2\rangle\\
        &=-2(i\partial_t (J_n v_n\cdot J_n u_n),J_n\partial_t u_n)_{L^2}+(J_n\partial_x^2 |J_n u_n|^2,\partial_t v_n)_{L^2}\\
        &=-2(i\partial_t J_n v_n\cdot J_n u_n,J_n\partial_t u_n)_{L^2}+(J_n\partial_t v_n,\partial_x^2 |J_n u_n|^2)_{L^2}\\
        &=-2(i\partial_t J_n v_n\cdot J_n u_n,J_n(i\partial_x^2 u_n-iJ_n(J_n v_n\cdot J_n u_n)))_{L^2}\\
        &\quad +(J_n\partial_t v_n,\partial_x^2 |J_n u_n|^2)_{L^2}\\
        &=2(\partial_t J_n v_n\cdot J_n u_n,-\partial_x^2 J_n u_n+J_{n}^{2}(J_n v_n\cdot J_n u_n))_{L^2}\\
        &\quad +(J_n\partial_t v_n,\partial_x^2 |J_n u_n|^2)_{L^2}\\
        &=(J_n\partial_t v_n,\partial_x^2 |J_n u_n|^2-2\mathrm{Re}(J_n \overline{u_n}\partial_x^2 J_n u_n))_{L^2}\\
        &\quad +2(\partial_t J_n v_n\cdot J_n u_n,J_{n}^{2}(J_n v_n\cdot J_n u_n))_{L^2}\\
        &=2(\partial_t J_n v_n,|\partial_x J_n u_n|^2+\mathrm{Re}(J_n \overline{u_n}J_{n}^2(J_n v_n\cdot J_n u_n)))_{L^2}.
    \end{align}
    where in the last equality we used
    \begin{equation}
        \partial_x^2|J_n u_n|^2-2\mathrm{Re}(J_n \overline{u_n}J_{n}^2(J_n v_n \cdot J_n u_n))=2|\partial_x J_n u_n|^2.
    \end{equation}
    We note that
    \begin{align}
        \|\partial_t u_n(t)\|_{L^2}^2
        &=\|\partial_x^2 u_n-J_n(J_n v_n\cdot J_n u_n)\|_{L^2}^2 \label{eq:3.2}\\
        &=\|\partial_x^2 u_n\|_{L^2}^2+(\partial_x J_n v_n,\partial_x |J_n u_n|^2)_{L^2}+2(J_n v_n,|\partial_x J_n u_n|^2)_{L^2}\\
        &\quad +\|J_n(J_n v_n\cdot J_n u_n)\|_{L^2}^2
    \end{align}
    and the inner product terms in \eqref{eq:3.2} are estimated as
    \begin{align}
        |(\partial_x J_n v_n,\partial_x (|J_n u_n|^2))|
        &\le 2\|\partial_x J_n v_n\|_{L^2}\|\partial_x J_n u_n J_n \overline{u_n}\|_{L^2}\label{eq:3.3}\\
        &\le 2\|\partial_x J_n v_n\|_{L^2}\|\partial_x J_n u_n\|_{L^4}\|J_n u_n\|_{L^4}\\
        &\le C\|\partial_x v_n\|_{L^2}\|u_n\|_{L^2}^{3/4}\|\partial_x u_n\|_{L^2}^{1/4}\|\partial_x u_n\|_{L^2}^{3/4}\|\partial_x^2 u_n\|_{L^2}^{1/4}\\
        &\le C(M_1)\|\partial_x v_n\|_{L^2}\|\partial_x^2 u_n\|_{L^2}^{3/4}\\
        &\le \left(\frac{1}{2}\|\partial_x^2 u_n\|_{L^2}+C(M_1)\right)\|\partial_x v_n\|_{L^2}\\
        &\le \frac{3}{8}\|\partial_x^2 u_n\|_{L^2}^2+\frac{1}{3}\|\partial_x v_n\|_{L^2}^2+C(M_1), 
    \end{align}
    Similarly, we have
    \begin{align}
        |(J_n v_n,|\partial_x J_n u_n|^2)_{L^2}|
        &\le \frac{1}{16}\|\partial_x^2 u_n\|_{L^2}^2+C(M_1)\label{eq:3.4}.
    \end{align}
    By \eqref{eq:3.3} and \eqref{eq:3.4}, we have
    \begin{align}
        \frac{1}{2}\|\partial_x^2 u_n\|_{L^2}^2
        &+\frac{1}{6}\|\partial_x v_n\|_{L^2}^2+\frac{1}{2}\|\partial_x w_n\|_{L^2}^2\label{eq:3.5}\\
        &\le F_n(t)+C(M_1)\\
        &\le \frac{3}{2}\|\partial_x^2 u_n\|_{L^2}^2+\frac{4}{3}\|\partial_x v_n\|_{L^2}^2+\frac{1}{2}\|\partial_x w_n\|_{L^2}^2+C(M_1).
    \end{align}
    From \eqref{eq:3.1} and \eqref{eq:3.5}, we obtain
    \begin{align}
        \left|\frac{d}{dt}F_n(t)\right| &\le 2\|\partial_t J_n v_n\|_{L^2}(\|\partial_x J_n u_n\|_{L^4}^2+\|J_n u_n\|_{L^{\infty}}\|J_n^2(J_n v_n \cdot J_n u_n)\|_{L^2})\label{eq:3.6}\\
        &\le 2\|\partial_t v_n\|_{L^2}(\|\partial_x u_n\|_{L^4}^2+\|J_n u_n\|_{L^{\infty}}^2\|J_n v_n\|_{L^2})\\
        &\le C\|\partial_t v_n\|_{L^2}(\|\partial_x u_n\|_{L^2}^{3/2}\|\partial_x^2 u_n\|_{L^2}^{1/2}+\|u_n\|_{H^1}^2\|v_n\|_{L^2})\\
        &\le C(M_1)\|\partial_t v_n\|_{L^2}(\|\partial_x^2 u_n\|_{L^2}^{1/2}+1)\\
        &\le C(M_1)(\|\partial_x^2 u_n\|_{L^2}^2+\|\partial_x w_n\|_{L^2}^2)^{3/4}+\|\partial_x w_n\|_{L^2}^{3/2}+C(M_1)\\
        &\le C(M_1)(F_n(t)+1)^{3/4}.
    \end{align}
    Therefore, by \eqref{eq:3.5} and \eqref{eq:3.6}, there exists $C_0=C_0(\|(u_0,v_0,w_0)\|_{H^2\times H^1\times H^1})>0$ such that
    \begin{align}
        F_n(t)&\le C(M_1)(F_n(0)+|t|^4)\\
        &\le C_0(1+|t|^4)=:M_2(t)^2.
    \end{align}
    Thus, we obtain the result.
\end{proof}

\section{Convergence of the sequence of approximate solutions}
Let $T>0$ be arbitrary. In previous section, we have
\begin{equation}
    \sup_{n\in\mathbb{N},t\in[-T,T]}\|(u_n(t),v_n(t),w_n(t))\|_{H^2\times H^1\times H^1}\le C_0(1+|T|^2)=:M_2.
\end{equation}

In this section, we prove the following proposition.

\begin{proposition}
    The sequence $((u_n,v_n,w_n))_{n\in\mathbb{N}}$ of the solutions to the approximate system \eqref{eq:(2.1)} is a Cauchy sequence of $C([-T,T];H_0^1(I)\times L^2(I)\times L^2(I))$.\label{prop_4_1}
\end{proposition}

Now, we prepare to prove Proposition \ref{prop_4_1}.

\begin{lemma}
    For all $t\in[-T,T]$, 
    \begin{equation}
        \frac{d}{dt}\|u_m(t)-u_n(t)\|_{L^2}^2
        \le C(M_2)\left(\|u_m(t)-u_n(t)\|_{L^2}^2+\|v_m(t)-v_n(t)\|_{L^2}^2+\frac{1}{n^{1/2}}\right).
    \end{equation}
\end{lemma}

\begin{proof}
Since 
\begin{equation}
        (iJ_n v_n\cdot J_m(u_m-u_n),J_m(u_m-u_n))_{L^2}=0,
\end{equation}
we have
    \begin{align}
        &\frac{d}{dt}\|u_m(t)-u_n(t)\|_{L^2}^2\\
        &=2(\partial_t (u_m-u_n),u_m-u_n)_{L^2}\\
        &=2(i\partial_x^2 (u_m-u_n)-i(J_m(J_m v_m\cdot J_m u_m)-J_n(J_n v_\cdot J_n u_n)),u_m-u_n)_{L^2}\\
        &=-2(iJ_m(v_m-v_n)\cdot J_m u_m,J_m(u_m-u_n))_{L^2}\\
        &\quad -2(i(J_m-J_n)v_n\cdot J_m u_m,J_m(u_m-u_n))_{L^2}\\
        &\quad -2(iJ_n v_n\cdot (J_m-J_n)u_n),J_m(u_m-u_n))_{L^2}\\
        &\quad -2(i(J_m-J_n)(J_n v_n\cdot J_n u_n),J_m(u_m-u_n))_{L^2}\\
        &\le 2\|v_m-v_n\|_{L^2}\|J_m u_m\|_{L^{\infty}}\|u_m-u_n\|_{L^2}\\
        &\quad +2\|(J_m-J_n)v_n\|_{L^4}\|J_m u_m\|_{L^4}\|u_m-u_n\|_{L^2})\\
        &\quad +2\|J_n v_n\|_{L^4}\|(J_m-J_n)u_n\|_{L^4}\|u_m-u_n\|_{L^2}\\
        &\quad +2\|(J_m-J_n)(J_n v_n\cdot J_n u_n)\|_{L^2}\|u_m-u_n\|_{L^2}\\
        &\le C\|v_m-v_n\|_{L^2}\|u_m\|_{H^2}\|u_m-u_n\|_{L^2}\\
        &\quad +C\frac{1}{n^{1/4}}\|\partial_x v_n\|_{L^2}\|u_m\|_{H^2}\|u_m-u_n\|_{L^2}\\
        &\quad +C\frac{1}{n^{1/4}}\|\partial_x u_n\|_{L^2}\|v_m\|_{H^1}\|u_m-u_n\|_{L^2}\\
        &\quad +C\frac{1}{n^{1/2}}(\|\partial_x J_n v_n\|_{L^2}\|J_n u_n\|_{L^{\infty}}+\|J_n v_n\|_{L^4}\|\partial_x J_n u_n\|_{L^4})\|u_m-u_n\|_{L^2}\\
        &=:I_1+I_2+I_3+I_4.
    \end{align}
    By Young's inequality,
    \begin{align}
        &
        \begin{aligned}
            I_1
            &=C\|v_m-v_n\|_{L^2}\|u_m\|_{H^2}\|u_m-u_n\|_{L^2}\\
            &\le C(M_2)\|v_m-v_n\|_{L^2}\|u_m-u_n\|_{L^2}\\
            &\le C(M_2)(\|v_m-v_n\|_{L^2}^2+\|u_m-u_n\|_{L^2}^2),
        \end{aligned}
        \\
        &
        \begin{aligned}
            I_2,I_3
            &\le C(M_2)\left(\|u_m-u_n\|_{L^2}^2+\frac{1}{n^{1/2}}\right),
        \end{aligned}
        \\
        &
        \begin{aligned}
            I_4
            &\le C\frac{1}{n^{1/4}}(\|\partial_x v_n\|_{L^2}\|u_n\|_{H^2}+\|v_n\|_{H^1}\|u_n\|_{H^2})\|u_m-u_n\|_{L^2}\\
            &\le C(M_2)\frac{1}{n^{1/4}}\|u_m-u_n\|_{L^2}\\
            &\le C(M_2)\left(\frac{1}{n^{1/2}}+\|u_m-u_n\|_{L^2}^2\right)
        \end{aligned}
    \end{align}
    We have the statement by collecting these estimates.
\end{proof}

We now consider the $H^1\times L^2\times L^2$ estimate of the sequence of the approximate solutions $(u_m(t)-u_n(t),v_m(t)-v_n(t),w_m(t)-w_n(t))$. For that purpose, we compute
\begin{align}
        \frac{d}{dt}&\left(\|\partial_x (u_m(t)-u_n(t))\|_{L^2}^2+\frac{1}{2}(\|v_m(t)-v_n(t))\|_{L^2}^2+\|w_m(t)-w_n(t)\|_{L^2})\right)\\
        &=-2(\partial_t (u_m-u_n),J_m(J_m v_m\cdot J_m u_m)-J_n(J_n v_n\cdot J_n u_n))_{L^2}\\
        &\quad +(J_m|J_m u_m|^2-J_n|J_n u_n|^2,-\partial_x(w_m-w_n))_{L^2}
\end{align}
where we have used
\begin{align}
&
\begin{aligned}
    2(\partial_t&\partial_x (u_m-u_n),\partial_x (u_m-u_n))_{L^2}\\
    &=-2(\partial_t (u_m-u_n),J_m(J_m v_m\cdot J_m u_m)-J_n(J_n v_n\cdot J_n u_n))_{L^2},
\end{aligned}
\\
&
\begin{aligned}
    (\partial_t &(v_m-v_n),(v_m-v_n))_{L^2}+(\partial_t (w_m-w_n),(w_m-w_n))_{L^2}\\
    &=(J_m|J_m u_m|^2-J_n|J_n u_n|^2,-\partial_x(w_m-w_n))_{L^2}.
\end{aligned}
\end{align}
We introduce the following modified energy:
\begin{align}
    \mathcal{E}_{m,n}(t)&=\|\partial_x (u_m-u_n)(t)\|_{L^2}^2+\frac{1}{2}(\|(v_m-v_n)(t)\|_{L^2}^2+\|(w_m-w_n)(t)\|_{L^2}^2)\\
    &\quad +2(J_m (u_m-u_n)(t),J_m(v_m-v_n)(t)\cdot J_m u_m(t))_{L^2}\\
    &\quad +2(J_m (u_m-u_n)(t),(J_m-J_n)v_n(t)\cdot J_m u_m(t))_{L^2}\\
    &\quad +(J_n v_n(t),|J_m(u_m-u_n)(t)|^2)_{L^2}\\
    &\quad +2(J_m (u_m-u_n)(t),J_n v_n(t)\cdot (J_m-J_n)u_n(t))_{L^2}\\
    &\quad +2((u_m-u_n)(t),(J_m-J_n)(J_n v_n(t)\cdot J_n u_n(t)))_{L^2}.
\end{align}
\begin{lemma}
    For all $t\in\R$,
    \begin{align}
        \frac{d}{dt}\mathcal{E}_{m,n}(t)
        \le C(M_2)\left(\|u_m(t)-u_n(t)\|_{H^1}^2+\|v_m(t)-v_n(t)\|_{L^2}^2+\|w_m(t)-w_n(t)\|_{L^2}^2+\frac{1}{n^{1/2}}\right).
    \end{align}
\end{lemma}
\begin{proof}
From the equality
\begin{align}
        (-\partial_x &(w_m-w_n),J_m|J_m u_m|^2-J_n|J_n u_n|^2)_{L^2}\\
        &=(-\partial_x (w_m-w_n),J_m(J_m(\overline{u_m}-\overline{u_n})J_m u_m))_{L^2}\\
        &\quad +(-\partial_x (w_m-w_n),J_m((J_m-J_n)\overline{u_n}\cdot J_m u_m))_{L^2}\\
        &\quad +(-\partial_x (w_m-w_n),J_m(J_n \overline{u_n}\cdot J_m(u_m-u_n)))_{L^2}\\
        &\quad +(-\partial_x (w_m-w_n),J_m(J_n\overline{u_n}\cdot(J_m-J_n)u_n)))_{L^2}\\
        &\quad +(-\partial_x (w_m-w_n),(J_m-J_n)|J_n u_n|^2)_{L^2},
\end{align}
we have
    \begin{align}
        \frac{d}{dt}\mathcal{E}_{m,n}(t)
        &=-2(J_m (u_m-u_n),-J_m\partial_x(w_m-w_n)\cdot J_m u_m)_{L^2}\\
        &\quad +2(J_m (u_m-u_n),J_m(v_m-v_n)\cdot J_m \partial_t u_m)_{L^2}\\
        &\quad -2(J_m (u_m-u_n),-(J_m-J_n)\partial_x w_n\cdot J_m u_m)_{L^2}\\
        &\quad +2(J_m (u_m-u_n),(J_m-J_n)v_n\cdot J_m \partial_t u_m)_{L^2}\\
        &\quad -2(J_m (u_m-u_n),-J_n\partial_x w_n\cdot (J_m-J_n)u_n)_{L^2}\\
        &\quad +2(J_m (u_m-u_n),J_n v_n\cdot (J_m-J_n)\partial_t u_n)_{L^2}\\
        &\quad -2(u_m-u_n,(J_m-J_n)(-J_n\partial_x w_n\cdot J_n u_n))_{L^2}\\
        &\quad +2(u_m-u_n,(J_m-J_n)(J_n v_n\cdot J_n \partial_t u_n))_{L^2}\\
        &\quad -(-\partial_x (w_m-w_n),J_m(J_m(\overline{u_m}-\overline{u_n})J_m u_m))_{L^2}\\
        &\quad -(-\partial_x (w_m-w_n),J_m((J_m-J_n)\overline{u_n}\cdot J_m u_m))_{L^2}\\
        &\quad -(-\partial_x (w_m-w_n),J_m(J_n \overline{u_n}\cdot J_m(u_m-u_n)))_{L^2}\\
        &\quad -(-\partial_x (w_m-w_n),J_m(J_n\overline{u_n}\cdot(J_m-J_n)u_n)))_{L^2}\\
        &\quad -(-\partial_x (w_m-w_n),(J_m-J_n)|J_n u_n|^2)_{L^2}\\
        &\quad -(-J_n\partial_x w_n,|J_m(u_m-u_n)|^2)_{L^2}\\
        &=:\sum_{j=1}^{14}II_j
    \end{align}
Here, we estimate each of $II_j$. By using Lemma \ref{lemma_2_1}, Lemma \ref{lemma_2_2}, Gagliardo--Nirenberg inequality and the Sobolev embedding, we have
\begin{align}
    |II_1|
    &=2|(J_m (u_m-u_n),-J_m\partial_x(w_m-w_n)\cdot J_m u_m)_{L^2}|\\
    &\le |2(-\partial_x (J_m(u_m-u_n)\cdot J_m\overline{u_m}),J_m(w_m-w_n))_{L^2}|\\
    &\le \|\partial_x (J_m(u_m-u_n)\cdot J_m\overline{u_m})\|_{L^2}\|w_m-w_n\|_{L^2}\\
    &\le C(\|\partial_x J_m(u_m-u_n)\|_{L^2}\|J_m u_m\|_{L^{\infty}}+\|J_m(u_m-u_n)\|_{L^4}\|\partial_x J_m u_m\|_{L^4})\|w_m-w_n\|_{L^2}\\
    &\le C\|u_m\|_{H^2}\|u_m-u_n\|_{H^1}\|w_m-w_n\|_{L^2}\\
    &\le C(M_2)\|u_m-u_n\|_{H^1}\|w_m-w_n\|_{L^2}.
\end{align}
A similar calculation shows
\begin{align}
    |II_2|,|II_9|,|II_{11}|&\le C(M_2)\|u_m-u_n\|_{H^1}\|v_m-v_n\|_{L^2},\\
    |II_{14}|&\le C(M_2)\|u_m-u_n\|_{H^1}^2
\end{align}
Moreover, we obtain
\begin{align}
    |II_3|
    &=|2(J_m (u_m-u_n),-(J_m-J_n)\partial_x w_n\cdot J_m u_m)_{L^2}|\\
    &\le 2|(-\partial_x (J_m (u_m-u_n)\cdot J_m \overline{u_m}),(J_m-J_n)w_n)_{L^2}|\\
    &\le 2|(\partial_x (J_m (u_m-u_n)\cdot J_m \overline{u_m}),(J_m-J_n)w_n)_{L^2}|\\
    &\le C(\|\partial_x (u_m-u_n)\|_{L^2}\|J_m u_m\|_{L^{\infty}}+\|J_m (u_m-u_n)\|_{L^4}\|\partial_x J_m u_m\|_{L^4})\|(J_m-J_n)w_n\|_{L^2}\\
    &\le C\frac{1}{n^{1/2}}\|u_m-u_n\|_{H^1}\|u_m\|_{H^2}\|\partial_x w_n\|_{L^2}\\
    &\le C(M_2)\frac{1}{n^{1/2}}\|u_m-u_n\|_{H^1}.
\end{align}
Estimating the remaining terms in the similar way as above, we have
\begin{align}
    |II_4|,|II_5|,|II_6|,|II_8|&\le C(M_2)\frac{1}{n^{1/4}}\|u_m-u_n\|_{H^1}\\
    |II_7|&\le C(M_2)\frac{1}{n^{1/2}}\|u_m-u_n\|_{H^1}\\
    |II_{10}|,|II_{12}|&\le C(M_2)\frac{1}{n^{1/4}}\|w_m-w_n\|_{L^2}\\
    |II_{13}|&\le C(M_2)\frac{1}{n^{1/2}}\|w_m-w_n\|_{L^2}.
\end{align}
Applying the Young's inequality, we obtain
\begin{align}
    &\left| \frac{d}{dt}\mathcal{E}_{m,n}(t) \right|\\
    &\le C(M_2)(\|u_m-u_n\|_{H^1}\|w_m-w_n\|_{L^2}+\|u_m-u_n\|_{H^1}\|v_m-v_n\|_{L^2}\\
    &\quad +\frac{1}{n^{1/4}}\|u_m-u_n\|_{H^1}+\frac{1}{n^{1/4}}\|w_m-w_n\|_{L^2}+\|u_m-u_n\|_{H^1}^2)\\
    &\le C(M_2)\left(\|u_m-u_n\|_{H^1}^2+\|v_m-v_n\|_{L^2}^2+\|w_m-w_n\|_{L^2}^2+\frac{1}{n^{1/2}}\right).
\end{align}
This is the desired inequality.
\end{proof}
We are going to estimate each of the inner product terms included in $\mathcal{E}_{m,n}$. By Lemmas $2.1$, $2.2$, the Gagliardo--Nirenberg and Young's inequality, we have
\begin{align}
    2|(J_m &(u_m-u_n),J_m(v_m-v_n)\cdot J_m u_m)_{L^2}|\\
    &\le 2\|J_m (u_m-u_n)\|_{L^4}\|J_m u_m\|_{L^4}\|J_m(v_m-v_n)\|_{L^2}\\
    &\le C\|u_m-u_n\|_{L^2}^{3/4}\|\partial_x (u_m-u_n)\|_{L^2}^{1/4}\|u_m\|_{L^2}^{3/4}\|\partial_x u_m\|_{L^2}^{1/4}\|v_m-v_n\|_{L^2}\\
    &\le C(M_1)\|u_m-u_n\|_{L^2}^2+\frac{1}{2}\|\partial_x (u_m-u_n)\|_{L^2}^{2/5}\|v_m-v_n\|_{L^2}^{8/5}\\
    &\le C(M_1)\|u_m-u_n\|_{L^2}^2+\frac{1}{4}(\|\partial_x u_m-u_n\|_{L^2}^2+\|v_m-v_n\|_{L^2}^2)\\
    &\le \frac{1}{4}\|\partial_x u_m-u_n\|_{L^2}^2+\frac{1}{4}\|v_m-v_n\|_{L^2}^2+C(M_1)\|u_m-u_n\|_{L^2}^2.
\end{align}
Similarly, we have the following inequalities:
\begin{align}
    &2|(J_m(u_m-u_n),(J_m-J_n)v_n\cdot J_m u_m)_{L^2}|\le C(M_2)\frac{1}{n^{1/2}},\\
    &|(J_n v_n,|J_m(u_m-u_n)|^2)_{L^2}|\le \frac{1}{4}\|\partial_x(u_m-u_n)\|_{L^2}^2+C(M_1)\|u_m-u_n\|_{L^2}^2,\\
    &|2(J_m (u_m-u_n),J_n v_n\cdot (J_m-J_n)u_n)_{L^2}|\le C(M_2)\frac{1}{n^{1/2}},\\
    &|2(u_m-u_n,(J_m-J_n)(J_n v_n\cdot J_n u_n))_{L^2}|\le \frac{1}{4}\|\partial_x(u_m-u_n)\|_{L^2}^2+C(M_1)\frac{1}{n^{1/2}}.
\end{align}
Therefore, we obtain
\begin{align}
    &\frac{1}{4}\|\partial_x(u_m-u_n)\|_{L^2}^2+\frac{1}{4}\|v_m-v_n\|_{L^2}^2+\frac{1}{2}\|w_m-w_n\|_{L^2}^2\label{eq:(4.1)}\\
    &\qquad -C(M_1)\|u_m-u_n\|_{L^2}^2-C(M_2)\frac{1}{n^{1/2}}\\
    &\qquad \qquad \le E_{m,n}(t)\\
    &\qquad \qquad \qquad \le \frac{7}{4}\|\partial_x(u_m-u_n)\|_{L^2}^2+\frac{5}{4}\|v_m-v_n\|_{L^2}^2+\frac{1}{2}\|w_m-w_n\|_{L^2}^2\\
    &\qquad \qquad \qquad \qquad +C(M_1)\|u_m-u_n\|_{L^2}^2+C(M_2)\frac{1}{n^{1/2}}.
\end{align}
We put
\begin{equation}
    \mathcal{F}_{m,n}(t)=\mathcal{E}_{m,n}(t)+C(M_1)\|(u_m-u_n)(t)\|_{L^2}^2+C(M_2)\frac{1}{n^{1/2}},
\end{equation}
then the following lemma holds.
\begin{lemma}
    For all $t\in\R$,
    \begin{equation}
        \mathcal{F}_{m,n}(t)\to 0\hspace{2mm}\mathrm{as}\hspace{2mm}m,n\to \infty.
    \end{equation}
\end{lemma}
\begin{proof}
Let $t\in\R$. By Lemma $4.2$, $4.3$ and \eqref{eq:(4.1)}, we obtain
\begin{align}
    \frac{d}{dt}\mathcal{F}_{m,n}(t)&=\frac{d}{dt}\mathcal{E}_{m,n}(t)+C(M_1)\frac{d}{dt}\|u_m-u_n\|_{L^2}^2\label{eq:(4.2)}\\
    &\le C(M_2)(\|u_m-u_n\|_{H^1}^2+\|v_m-v_n\|_{L^2}^2+\|w_m-w_n\|_{L^2}^2)+C(M_2)\frac{1}{n^{1/2}}\\
    &\le C(M_2)\mathcal{F}_{m,n}(t).
\end{align}
From \eqref{eq:(4.2)} we have
\begin{equation}
    \mathcal{F}_{m,n}(t)\le \mathcal{F}_{m,n}(0)\exp{(C(M_2)|t|)}\to 0\hspace{2mm}\mathrm{as}\hspace{2mm}n,m\to\infty,
\end{equation}
where we have used
\begin{equation}
    \mathcal{F}_{m,n}(0)\to 0\hspace{2mm}\mathrm{as}\hspace{2mm}m,n\to\infty.
\end{equation}
\end{proof}
Now, we show Proposition \ref{prop_4_1}.
\begin{proof}[Proof of Proposition \ref{prop_4_1}]
For all $|t|\le T$, using Lemma $4.2$, $4.3$ and $4.4$
\begin{align}
    \|u_m(t)-u_n(t)\|_{H^1}^2+
    &\|v_m(t)-v_n(t)\|_{L^2}^2+\|w_m(t)-w_n(t)\|_{L^2}^2\\
    &\le C(M_2)F_{m,n}(t)\\
    &\le C(M_2)F_{m,n}(0)\exp{(C(M_2)|T|)}\to 0.
\end{align}
\end{proof}

\section{Proof of the main theorem}
First, we show Theorem \ref{theorem1}.
From Proposition $4.1$ and the completeness of $C(\R;H_0^1(I)\times L^2(I)\times L^2(I))$, we obtain $(u,v,w)\in C(\R;H_0^1(I)\times L^2(I)\times L^2(I))$ as the limit of the sequence $(u_n,v_n,w_n)$.
\begin{proposition}
    For all $t\in\R$, the following conservation laws hold;
    \begin{align}
        \|u(t)\|_{L^2}=\|u_0\|_{L^2},\hspace{1mm}E(U(t))=E(U_0).
    \end{align}
\end{proposition}
\begin{proof}
It follows immediately from Lemma $2.3$.
\end{proof}
\begin{proposition}
There exists $C=C(\|(u_0,v_0,w_0)\|_{H^2\times H^1\times H^1})>0$ such that for all $t\in\R$,
\begin{equation}
    \|u(t)\|_{H^2}+\|v(t)\|_{H^1}+\|w(t)\|_{H^1}\le C(1+|t|^2).
\end{equation}
\end{proposition}

\begin{proof}
Since for all $\psi\in (H^2\cap H_0^1)(I)$,
\begin{align}
    (u(t),\partial_x^2\psi)_{L^2}&=\lim_{n\to\infty}(u_n(t),\partial_x^2\psi)_{L^2}
    =\lim_{n\to\infty}(\partial_x^2 u_n(t),\psi)_{L^2}
    \le \|\partial_x^2 u_n(t)\|_{L^2}\|\psi\|_{L^2}
    \le M_2(t)\|\psi\|_{L^2},
\end{align}
we have $\partial_x^2 u(t)\in L^2(I)$ for all $t\in\R$ and by density
\begin{equation}
    |(\partial_x^2 u(t),\psi)_{L^2}|\le M_2(t)\|\psi\|_{L^2}\hspace{2mm}\mathrm{for}\hspace{2mm}\mathrm{all}\hspace{2mm}\psi\in L^2(I),
\end{equation}
where $M_2(t)=C(1+|t|^2)$.
Therefore, we obtain
\begin{equation}
    \|\partial_x^2 u(t)\|\le M_2(t)\|\psi\|_{L^2}.
\end{equation}
From $u_n(t)\rightharpoonup u(t)\hspace{2mm}\mathrm{in}\hspace{2mm}L^2(I)$, we have
\begin{align}
    \|\partial_x^2 u(t)\|_{L^2}&\le \liminf_{n\to\infty}\|\partial_x^2 u_n(t)\|_{L^2}\le M_2(t).
\end{align}
Moreover, in the similar way to the above, we obtain
\begin{align}
    &\|\partial_x v(t)\|_{L^2}\le \liminf_{n\to\infty}\|\partial_x v_n(t)\|_{L^2}\le M_2(t),\\
    &\|\partial_x w(t)\|_{L^2}\le \liminf_{n\to\infty}\|\partial_x w_n(t)\|_{L^2}\le M_2(t).
\end{align}
The statement follows from Proposition \ref{prop_3_1}.
\end{proof}
We can complete the proof of Theorem $1.3$.

\begin{proposition}
$(u,v,w)$ satisfies the system \eqref{eq:(1.2)} and
\begin{align}
    &u\in (L_{loc}^{\infty}\cap C_w)(\R;(H^2\cap H_0^1)(I))\cap C(\R;H_0^1(I))\cap C^1(\R;H^{-1}(I)),\\
    &\partial_t u\in(L_{loc}^{\infty}\cap C_w)(\R;L^2(I)),\\
    &v\in C(\R;H_0^1(I))\cap C^1(\R;L^2(I)),\\
    &w\in C(\R;H_0^1(I))\cap C^1(\R;L^2(I)).
\end{align}
\end{proposition}

\begin{proof}
Let $U(t)$ be a group generated by
\begin{equation}
\begin{pmatrix*}[c]
    0&-\partial_x\\
    -\partial_x&0
\end{pmatrix*},
\end{equation}
and put
\begin{equation}
H_n(s)=
\begin{pmatrix*}[c]
    0\\
    -J_n\partial_x(|J_n u(s)|^2)
\end{pmatrix*},\quad
H(s)=
\begin{pmatrix*}
    0\\
    -\partial_x(|u(s)|^2)
\end{pmatrix*}.
\end{equation}
Here we set
$V_n(t)=(v_n(t),w_n(t))$. Then $V_n(t)$ satisfies the following integral equation:
\begin{equation}
    V_n(t)=U(t)V_n(0)+\displaystyle\int_{0}^{t}U(t-s)H_n(s)ds.
\end{equation}
Since the group $U(t)$ is a unitary group in $L^2$ and
\begin{equation}
    \sup_{t\in[-T,T]}\|J_n\partial_x(|J_n u(t)|^2)-\partial_x(|u(t)|^2)\|_{L^2}\to 0\hspace{2mm}\mathrm{as}\hspace{2mm}n\to\infty,
\end{equation}
$V(t)=(v(t),w(t))$ satisfies
\begin{equation}
    V(t)=U(t)V(0)+\displaystyle\int_{0}^{t}U(t-s)H(s)ds.
\end{equation}
We have $H\in L_{loc}^{\infty}(\R;H_0^1(I)\times H_0^1(I))$ by $\partial_x(|u(t)|^2)\in L_{loc}^{\infty}(\R;H_0^1(I))$. Since we also have the initial data $V(0)\in H_0^1(I)\times H_0^1(I)$, we obtain $V\in C(\R;H_0^1(I)\times H_0^1(I))$.

Next, we show that $(u,v,w)$ satisfies the system \eqref{eq:(1.2)}.
It follows from \eqref{eq:(2.1)} for any $\phi\in \mathcal{D}(-T,T)$ and $\psi\in H_0^1(I)$,
\begin{equation}
\label{eq:5.1}
    \displaystyle\int_{[-T,T]}\{-\langle iu_n(t),\psi\rangle \phi'(t)+\langle \partial_x^2 u_n(t)-J_n(J_n v_n\cdot J_n u_n)(t),\psi\rangle\phi(t)\}dt=0.
\end{equation}
Note that
\begin{equation}
    \sup_{t\in[-T,T]}\|J_n(J_n v_n\cdot J_n u_n)(t)-v(t)u(t)\|_{L^2}\to 0.\label{eq:5.2}
\end{equation}
Applying \eqref{eq:5.2} and the dominated convergence theorem to \eqref{eq:5.1}, we have
\begin{equation}
    \displaystyle\int_{[-T,T]}\{-\langle iu(t),\psi\rangle\phi'(t)+\langle\partial_x^2 u(t)-vu(t),\psi\rangle\phi(t)\}dt=0.
\end{equation}
This implies that for $t\in[-T,T]$
\begin{equation}
    -i\partial_t u(t)=\partial_x^2 u(t)+v(t)u(t)\hspace{2mm}\mathrm{in}\hspace{2mm}H^{-1}(I).
\end{equation}
Similarly, we obtain
\begin{equation}
    \left\{\begin{alignedat}{2}
    &\partial_t v=\partial_x w,\\
    &\partial_t w=-\partial_xv-\partial_x(|u|^2)
\end{alignedat}\right.
\end{equation}
in the distribution sense.

Furthermore, we obtain $u\in C_w(\R;(H^2\cap H_0^1)(I))$ by $u\in L_{loc}^{\infty}(\R;(H^2\cap H_0^1)(I))$.
\end{proof}

\begin{proposition}
    The solutions of \eqref{eq:(1.2)} are unique with respect to the initial values.
\end{proposition}
\begin{proof}
It suffices to check $U_1=U_2$ where $U_j=(u_j,v_j,w_j)\hspace{1mm}(j=1,2)$ is a solution of \eqref{eq:(1.2)}. We obtain the uniqueness in the same way as in the proof of Proposition $4.1$(see \cite{MR1}). 
\end{proof}
\begin{proposition}
$u$ is the strong solution, that is,
\begin{equation}
    u\in C(\R;(H^2\cap H_0^1)(I)).
\end{equation}
\end{proposition}
\begin{proof}
By the uniqueness of the solution, it suffices to check the continuity at $t=0$. Since we have
\begin{align}
    &\partial_t u(t)\rightharpoonup \partial_t u(0)\ \text{weakly}\ \text{in}\ L^2(I)\ \text{as}\ t\to 0,\\
    &\|\partial_t u(0)\|_{L^2}\le\liminf_{t\to 0}\|\partial_t u(t)\|_{L^2},
\end{align}
we have to check
\begin{equation}
    \limsup_{t\to 0}\|\partial_t u(t)\|_{L^2}\le \|\partial_t u(0)\|_{L^2}.
\end{equation}
By \eqref{eq:3.6}, for $|t|\le T$,
\begin{align}
    &\frac{d}{dt}\left\{\|\partial_t u_n(t)\|_{L^2}^2+\frac{1}{2}\left(\|\partial_x v_n(t)\|_{L^2}^2+\|\partial_x w_n(t)\|_{L^2}^2\right)\right\}\\
    &\le C(M_1)(\mathcal{F}_n(t)+1)\\
    &\le C(M_1)M_2.
\end{align}
Integrating from $0$ to $t$, we have
\begin{align}
    &\|\partial_t u_n(t)\|_{L^2}^2+\frac{1}{2}(\|\partial_x v_n(t)\|_{L^2}^2+\|\partial_x w_n(t)\|_{L^2}^2)\\
    &\le \|\partial_t u_n(0)\|_{L^2}^2+\frac{1}{2}(\|\partial_x v_n(0)\|_{L^2}^2+\|\partial_x w_n(0)\|_{L^2}^2)+Ct.
\end{align}
Therefore, we obtain
\begin{align}
    &\|\partial_t u(t)\|_{L^2}^2+\frac{1}{2}(\|\partial_x v(t)\|_{L^2}^2+\|\partial_x w(t)\|_{L^2}^2)\label{eq:5.3}\\
    &\le \liminf_{n\to \infty}\left\{\|\partial_t u_n(t)\|_{L^2}^2+\frac{1}{2}(\|\partial_x v_n(t)\|_{L^2}^2+\|\partial_x w_n(t)\|_{L^2}^2)\right\}\\
    &\le \liminf_{n\to\infty}\left\{\|\partial_t u_n(0)\|_{L^2}^2+\frac{1}{2}(\|\partial_x v_n(0)\|_{L^2}^2+\|\partial_x w_n(0)\|_{L^2}^2)\right\}+Ct\\
    &\le \|\partial_t u(0)\|_{L^2}^2+\frac{1}{2}(\|\partial_x v_0\|_{L^2}^2+\|\partial_x w_0\|_{L^2}^2)
\end{align}
where in the last inequality we have used
\begin{equation}
    \partial_t u_n(0)\to\partial_t u(0)\hspace{2mm}\mathrm{in}\hspace{2mm}L^2(I)\hspace{2mm}\mathrm{as}\hspace{2mm}n\to\infty.
\end{equation}
Furthermore, taking the $\limsup_{t\to0}$ in \eqref{eq:5.3} with attention to $v,w\in C(\R;H_0^1(I))$, we obtain
\begin{equation}
    \limsup_{t\to 0}\|\partial_t u(t)\|_{L^2}^2\le \|\partial_t u(0)\|_{L^2}^2.
\end{equation}
\end{proof}
The continuous dependence is verified by the same argument as in the proof of the uniqueness. Then we can complete the proof of the main theorems.

\section*{Acknowledgments}
The author would like to thank Professor Masahito Ohta, Dr. Noriyoshi Fukaya, Dr. Yoshinori Nishii and Mr. Kaito Kokubu for their support and helpful discussions for this study.

\bibliography{citation}
\bibliographystyle{amsplain}
\end{document}